\documentclass[oneside,english,reqno]{amsart}
\usepackage{amsthm}
\usepackage{amstext}
\usepackage{amssymb}

\makeatletter
\numberwithin{equation}{section}
\numberwithin{figure}{section}
\theoremstyle{plain}
\newtheorem{thm}{\protect\theoremname}[section]
  \theoremstyle{definition}
  \newtheorem{defn}[thm]{\protect\definitionname}
  \theoremstyle{remark}
  \newtheorem{rem}[thm]{\protect\remarkname}
  \theoremstyle{plain}
  \newtheorem{lem}[thm]{\protect\lemmaname}
  \theoremstyle{plain}
  \newtheorem{cor}[thm]{\protect\corollaryname}
  \theoremstyle{plain}
  \newtheorem{prop}[thm]{\protect\propositionname}

\makeatother

\usepackage{babel}
  \providecommand{\corollaryname}{Corollary}
  \providecommand{\definitionname}{Definition}
  \providecommand{\lemmaname}{Lemma}
  \providecommand{\propositionname}{Proposition}
  \providecommand{\remarkname}{Remark}
\providecommand{\theoremname}{Theorem}

\begin{document}

\title{Reduced $L_{q,p}$-Cohomology of Some  Twisted Products}

\author{Vladimir Gol$'$dshtein}
\address{Department of Mathematics, Ben Gurion University of the Negev, P.O.B. 653, Beer-Sheva 84105,
Israel}
\email{vladimir@bgu.ac.il}
\author{Yaroslav Kopylov}
\address{Sobolev Institute of Mathematics, Pr. Akad. Koptyuga 4,
630090, Novosibirsk, Russia and  Novosibirsk State University,
ul.~Pirogova 2, 630090, Novosibirsk, Russia }
\email{yakop@math.nsc.ru}
\thanks{The second-named author was partially supported by the State Maintenance Program 
for the Leading Scientific Schools and Junior Scientists of the Russian Federation 
(Grant~NSh--2263.2014.1)}

\begin{abstract}
Vanishing results for reduced $L_{p,q}$-cohomology are established in the case 
of twisted products, which are a~generalization of warped products. Only the case 
$q \leq p$ is considered. This is an extension of some results by Gol$'$dshtein, 
Kuz$'$minov  and Shvedov about the $L_{p}$-cohomology of warped cylinders. 
One of the main observations is the vanishing of the ``middle-dimensional'' cohomology 
for a large class of manifolds. 


\end{abstract}

\maketitle

\section{Introduction}

The \textit{$L_{q,p}$-cohomology} $H_{q,p}^{k}(M)$ of a Riemannian manifold $(M,g)$ is defined 
to be the quotient of the space of closed $p$-integrable differential forms by the exterior 
differentials of $q$-integrable forms. The quotient space of $H_{q,p}^{k}(M)$ 
by the closure of zero is called the \textit{reduced $L_{q,p}$-cohomology} $\overline{H}_{q,p}^{k}(M)$. 

In~\cite{GT2009}, Gol$'$dshtein and Troyanov established some nonvanishing results 
for the reduced $L_{q,p}$-cohomology $\overline{H}_{q,p}^{k}(M)$ of~$M$
in the case of simply-connected complete manifolds with negative curvature.
They obtained sufficient conditions for the nonvanishing of the space $\overline{H}_{q,p}^{k}(M)$
when $M$ is a~Cartan--Hadamard manifold (that is, a~complete 
simply-connected Riemannian manifold of nonpositive sectional curvature). Every Cartan-Hadamard
manifold is bi-Lipschitz equivalent to a so-called twisted product outside any geodesic ball.

A \textit{twisted product} $X\times_{h}Y$ of two Riemannian manifolds 
$(X,g_{X})$ and $(Y,g_{Y})$ is the direct product manifold $X\times_{g}Y$ endowed
with a~Riemannian metric of the form 
\begin{equation}
g:=g_{X}+h^{2}(x,y)g_{Y},\label{tp-metric}
\end{equation}
where $h:X\times Y\to\mathbb{R}$ is a~smooth positive function (see
\cite{Chen81}). If $X$ is a~half-interval $[a,b[$ then the twisted product $X\times_{h}Y$ 
is called a~{\it twisted cylinder}.
                                                                                                    
We refer to an $m$-dimensional Riemannian manifold $(M,g_M)$ as 
an~\textit{asymptotic twisted product} (respectively, as an {\it asymptotic twisted cylinder})
if, outside an $m$-dimensional compact submanifold,
it is bi-Lipschitz equivalent to a~twisted product (respectively, to a~twisted cylinder). 
A~Cartan--Hadamard manifold is an asymptotic twisted cylinder of this type. Other examples 
are manifolds with flat ends, manifolds with cuspidal singularities etc.

In this paper, we prove some vanishing results for the (reduced) $L_{q,p}$-cohomology
of twisted cylinders $[a,b)\times_{h}N$ for a~positive smooth function
$h:[a,b)\times N\to\mathbb{R}$ in the case where the base $N$ is a closed manifold and 
$p\ge q>1$. 

If in~(\ref{tp-metric}) the function~$h$ depends only on~$x$ then we obtain the familiar notion 
of a~\textit{warped product} (see \cite{BiON}). Twisted products were the object of recent 
investigations~\cite{BDDO2012,DBZ2012,Fa2013,FGKU2001,KJKP2005,PoRe93}.
The reduced $L_{q,p}$-cohomology of warped cylinders $[a,b)\times_{h}N$,
i.e., of product manifolds $[a,b)\times N$ endowed with a warped product
metric 
\[
g=dt^{2}+h^{2}(t)g_{N},
\]
where $g_{N}$ is the Riemannian metric of~$N$ and $h:[a,b)\to\mathbb{R}$
is a positive smooth function, was studied by Gol$'$dshtein, Kuz$'$minov,
and Shvedov \cite{GKSh90_1}, Kuz$'$minov and Shvedov \cite{KSh93,KSh96}
(for $p=q$), and Kopylov \cite{Kop07} for $p,q\in[1,\infty)$, 
$\frac{1}{q}-\frac{1}{p}<\frac{1}{\dim N+1}$.

The main results of this paper are technical. Here we mention a ``universal'' consequence of 
the main results on the vanishing of the ``middle-dimensional'' cohomology:

\medskip
{\it Let $N$ be a closed smooth $n$-dimensional Riemannian manifold.
If $p\ge q>1$, $b=\infty$, and $\frac{n}{p}$
is an integer, then \hbox{$\overline{H}_{q,p}^{\frac{n}{p}}([0,\infty)\times_{h} N)=0$}. 

In particular, if $1<q\le 2$ and $n$ is even then
\hbox{$\overline{H}_{q,2}^{\frac{n}{2}}([0,\infty)\times_{h} N)=0$}.
}
\medskip

The result was not known even for $L_2$-cohomology. It does not depend on the~type of the warped 
Riemannian metric. Of course, the result leads to the~vanishing of the ``middle-dimensional'' cohomology 
for asymptotic twisted cylinders.

\smallskip
For Cartan--Hadamard manifolds, the main results imply: 

\medskip
{\it 
Let $(M,g)$ be an $m$-dimensional Cartan-Hadamard manifold (that is, 
a complete simply-connected Riemannian manifold of nonpositive sectional 
curvature) with sectional curvature $K\leq-1$ and Ricci curvature 
$Ric\ge-(1+\epsilon)^{2}(m-1)$. If $p\ge q>1$ and $\frac{m-1}{p}$ 
is an integer then $\overline{H}_{q,p}^{\frac{m-1}{p}+1}(M,g)=0$.

In particular,
\begin{itemize} 
\item if $1<q\le 2$ and $m$ is odd then \textup{$\overline{H}_{q,2}^{\frac{m-1}{2}+1}(M,g)=0;$}
\item if $p >1$ and $\frac{m-1}{p}$ is an~integer then 
$\overline{H}_{p}^{\frac{m-1}{p}+1}(M,g)=\overline{H}_{p'}^{\frac{m-1}{p'}}(M,g)=0;$ 
\item if $m$ is odd then \textup{$\overline{H}_{2}^{\frac{m-1}{2}}(M,g)=\overline{H}_{2}^{\frac{m+1}{2}}(M,g)=0$.}
\end{itemize}
}

\section{Basic Definitions}

In this section, we recall the main definitions and notations.

In what follows, we tacitly assume all manifolds to be oriented.

Let $M$ be a smooth Riemannian manifold. Denote by $\mathcal{D}^{k}(M):=C_{0}^{\infty}(M,\Lambda^{k})$
the space of all smooth differential $k$-forms with compact support
contained in~\mbox{$M\setminus\partial M$} and designate as $L_{loc}^{1}(M,\Lambda^{k})$
the space of locally integrable differential forms. Denote by $L^{p}(M,\Lambda^{k})$
the Banach space of locally integrable differential $k$-forms endowed
with the norm $\|\theta\|_{L^{p}(M,\Lambda^{k})}:=\left(\int_{M}|\theta|^{p}dx\right)^{\frac{1}{p}}<\infty$
(as usual, we identify forms coinciding outside a~set of measure
zero).
\begin{defn}
We call a~differential $(k+1)$-form $\theta\in L_{loc}^{1}(M,\Lambda^{k+1})$
\emph{the weak exterior derivative} (or \emph{differential}) of a
differential $k$-form $\phi\in L_{loc}^{1}(M,\Lambda^{k})$ and write
$d\phi=\theta$ if 
\[
\int_{M}\theta\wedge\omega=(-1)^{k+1}\int_{M}\phi\wedge d\omega
\]
 for any $\omega\in\mathcal{D}^{n-k}(M)$. \end{defn}
\begin{rem}
Note that the orientability of $M$ is not substantial since one may
take integrals over orientable domains on $M$ instead of integrals
over $M$.  

We then introduce an analog of Sobolev spaces for differential
$k$-forms, the space of $q$-integrable forms with $p$-integrable
weak exterior derivative: 
\[
\Omega_{q,p}^{k}(M)=\left\{ \,\omega\in L^{q}(M,\Lambda^{k})\;|\, d\omega\in L^{p}(M,\Lambda^{k+1})\right\} ,
\]
 This is a Banach space for the graph norm 
\[
\|\omega\|_{\Omega_{q,p}^{k}(M)}=\left(\|\omega\|_{L^{q}(M,\Lambda^{k})}^{2}+\|d\omega\|_{L^{p}(M,\Lambda^{k+1})}^{2}\right)^{1/2}.
\]
 The space $\Omega_{q,p}^{k}(M)$ is a reflexive Banach space for
any $1<q,p<\infty$. This can be proved using standard arguments of
functional analysis.
\end{rem}
Denote by $\Omega_{q,p,0}^{k}(M)$ the closure of $\mathcal{D}^{k}(M)$
in the norm of $\Omega_{q,p}^{k}(M)$. We now define our basic ingredients
(for three parameters $r,q,p$).
\begin{defn}
Put

(a) $Z_{p,r}^{k}(M)=\mathrm{Ker}[d:\Omega_{p,r}^{k}(M)\to L^{r}(M,\Lambda^{k+1})]$.

(b) $B_{q,p}^{k}(M)=\mathrm{Im}[d:\Omega_{q,p}^{k-1}(M)\to L^{p}(M,\Lambda^{k})]$.\end{defn}
\begin{lem}
The subspace $Z_{p,r}^{k}(M)$ does not depend on $r$ and is a closed
subspace in $L^{p}(M,\Lambda^{k})$. \end{lem}
\begin{proof}
The lemma is in fact \cite[Lemma~2.4~(i)]{GT2012}. However, we now
repeat the proof for the reader's convenience. Note that $Z_{p,r}^{k}(M)$
is a closed subspace in $\Omega_{p,r}^{k}(M)$ because it is the kernel
of the bounded operator $d$. It is also a closed subspace of $L^{p}(M,\Lambda^{k})$
since $\|\alpha\|_{\Omega_{p,r}^{k}(M)}=\|\alpha\|_{L^{p}(M,\Lambda^{k})}$
for any $\alpha\in Z_{p,r}^{k}(M)$. 
\end{proof}
This allows us to use the notation $Z_{p}^{k}(M)$ for all $Z_{p,r}^{k}(M)$.
Note that $Z_{p}^{k}(M)\subset L^{p}(M,\Lambda^{k})$ is always a
closed subspace but this is in general not true for $B_{q,p}^{k}(M)$.
Denote by $\overline{B}_{q,p}^{k}(M)$ its closure in the $L^{p}$-topology.
Observe also that since $d\circ d=0$, one has $\overline{B}_{q,p}^{k}(M)\subset Z_{p}^{k}(M)$.
Thus, 
\[
B_{q,p}^{k}(M)\subset\overline{B}_{q,p}^{k}(M)\subset Z_{p}^{k}(M)=\overline{Z}_{p}^{k}(M)\subset L^{p}(M,\Lambda^{k}).
\]

\begin{defn}
Suppose that $1\leq p,q\leq\infty$. The \emph{$L_{q,p}$-cohomology}
of $(M,g)$ is defined as the quotient 
\[
H_{q,p}^{k}(M):=Z_{p}^{k}(M)/B_{q,p}^{k}(M)\,,
\]
 and the \emph{reduced $L_{q,p}$-cohomology} of $(M,g)$ is, by definition,
the space 
\[
\overline{H}_{q,p}^{k}(M):=Z_{p}^{k}(M)/\overline{B}_{q,p}^{k}(M)\,.
\]

\end{defn}

Since $B_{p,q}^{k}$ is not always closed, the $L_{q,p}$-cohomology
is in general a (non-Haus\-dorff) semi-normed space, while the reduced
$L_{q,p}$-cohomology is a Banach space. Considering only the forms
equal to zero on some neighborhood (depending on the form) of a subset
$A\subset M$ and taking closures in the corresponding spaces, we
obtain the definition of the relative spaces $L^{p}(M,A,\Lambda^{k})$
and $\Omega_{q,p}(M,A)$ and the \emph{relative nonreduced} and \emph{reduced
cohomology spaces} $H_{q,p}^{k}(M,A)$ and $\overline{H}_{q,p}^{k}(M,A)$.

Similarly, one can define the $L_{q,p}$-cohomology with compact support (interior cohomology)  
$H_{q,p;0}^{k}(M,A)$ and $\overline{H}_{q,p;0}^{k}(M,A)$. The interior 
reduced cohomology is dual to the reduced cohomology:

\begin{thm}\label{thm:duality} \cite{GT2012}
Let $(M,g)$ be an oriented $n$-dimensional Riemannian manifold. If $1<p,q<\infty$ then 
 $\overline{H}_{q,p}^{k}(M)$ is isomorphic to the dual of $\overline{H}^{n-k}_{p',q';0}(M)$,
where $\frac{1}{p'}+\frac{1}{p}=\frac{1}{q'}+\frac{1}{q}=1$. 
\end{thm} 

Therefore, if $M$ is complete then $\overline{H}^k_{p}(M) = \overline{H}^k_{p;0}(M)$ 
and $\overline{H}^k_{p}(M)$ is isomorphic to dual of $\overline{H}^{n-k}_{p'}(M)$,
where $\frac{1}{p}+ \frac{1}{p'}  =  1$~\cite{GT2012}.

Below $|X|$ stands for the volume of a Riemannian manifold $(X,g)$.

\section{$L_{q,p}$-Cohomology and Smooth Forms}

It follows from the results of~\cite{GT2006} that, under suitable
assumptions on $p,q$, the $L_{q,p}$-cohomology of a Riemannian manifold
can be expressed in terms of smooth forms.

Introduce the notations: 
\begin{gather*}
C^{\infty}L^{p}(M,\Lambda^{k}):=C^{\infty}(M,\Lambda^{k})\cap L^{p}(M,\Lambda^{k});\\
C^{\infty}\Omega_{q,p}^{k}(M):=C^{\infty}(M,\Lambda^{k})\cap\Omega_{q,p}^{k}(M);\\
\end{gather*}

\begin{thm}
\label{sm-cohom} \emph{\cite[Theorem~12.8 and Corollary~12.9]{GT2006}.}
Let $(M,g)$ be an $m$-dimensional Riemannian manifold and suppose
that $p,q\in(1,\infty)$ satisfy $\frac{1}{p}-\frac{1}{q}\leq\frac{1}{m}$.
Then the cohomology $H_{q,p}^{*}(M)$ can be represented by smooth
forms.

More precisely, any closed form in $Z_{p}^{k}(M)$ is cohomologous
to a smooth form in $L^{p}(M)$. Furthermore, if two smooth closed
forms $\alpha,\beta\in C^{\infty}(M)\cap Z_{p}^{k}(M)$ are cohomologous
modulo $d\Omega_{q,p}^{k-1}(M)$ then they are cohomologous modulo
$dC^{\infty}\Omega_{q,p}^{k-1}(M)$.

Similarly, any reduced cohomology class can be represented by a smooth
form. 
\end{thm}

\medskip{}
In what follows, unless otherwise specified, we always assume that
$p,q\in(1,\infty)$ and $\frac{1}{p}-\frac{1}{q}\leq\frac{1}{\dim M}$.

\section{The Homotopy Operator}

From now on, $C_{a,b}^{h}N$ is the twisted cylinder $I\times_{h}N$,
that is, the product of a half-interval $I:=[a,b)$ and a~closed
smooth $n$-dimensional Riemannian manifold $(N,g_{N})$ equipped
with the Riemannian metric $dt^{2}+h^{2}(t,x)g_{N}$, where $h:I\times N\to\mathbb{R}$
is a~smooth positive function.

Every differential form on~$I\times N$ admits a unique representation
of the form $\omega=\omega_{A}+dt\wedge\omega_{B}$, where the forms
$\omega_{A}$ and $\omega_{B}$ do not contain $dt$ (cf. \cite{GKSh90_1}).
It means that $\omega_{A}$ and $\omega_{B}$ can be viewed as one-parameter
families $\omega_{A}(t)$ and $\omega_{B}(t)$, $t\in I$, of differential
forms on $N$. Given a form $\omega$ defined on $I\times N$ and
numbers $c,t\in[a,b)$, consider the form $\int_{c}^{t}\omega=\int_{c}^{t}\omega_{B}(\tau)d\tau$
on~$N$ and the form $S_{c}\omega$ on $[a,b)\times N$, $(S_{c}\omega)(t)=\int_{c}^{t}\omega$
for all $t\in[a,b)$. The domains of of these operators will be specified
below.

The modulus of a form $\omega$ of degree $k$ on $C_{a,b}^{h}N$
is expressed via the moduli of $\omega_{A}(t)$ and $\omega_{B}(t)$
on $N$ as follows: 
\begin{equation}
|\omega(t,x)|_{M}=\bigl[h^{-2k}(t,x)|\omega_{A}(t,x)|_{N}^{2}+h^{-2(k+1)}(t,x)|\omega_{B}(t,x)|_{X}^{2}\bigr]^{1/2}\label{eq:modulus}
\end{equation}

Consequently, 
\begin{multline}\label{eq:norm}
\|\omega\|_{L^p(C_{a,b}^h N,\Lambda^k)} \\
=\left[\!\int_{a}^{b}\int_{N}\bigl(h^{2(\frac{n}{p}-k)}(t,x) |\omega_{A}(t,x)|_{N}^{2}\!
+ h^{2(\frac{n}{p}-k+1)}(t,x) |\omega_{B}(t,x)|_{N}^{2}\bigr)^{\frac{p}{2}} 
dxdt\!\right]^{\frac{1}{p}}.
\end{multline}

In the particular case of $\omega=\omega_{A}$, call the form $\omega$
\emph{horizontal}. If $\omega$ is a horizontal form then 
\begin{equation}\label{eq:shortnorm}
\|\omega\|_{L^p(C_{a,b}^h N,\Lambda^k)}
=\left[\int_{a}^{b}\int_{N} |\omega(t)|^p h^{n-kp}(t,x) dxdt\right]^{1/p}.
\end{equation}

Put 
\[
f_{k,p}(t)=\min_{x\in N}\bigl\{ h^{\frac{n}{p}-k}(t,x)\bigr\}
\]
 and 
\[
F_{k,p}(t)=\max_{x\in N}\bigl\{ h^{\frac{n}{p}-k}(t,x)\bigr\}.
\]

\begin{rem}
\label{rem: absolute}
(1) Suppose that $k=\frac{n}{p}$ is an integer.
Then $F_{n/p,p}(t)=f_{n/p,p}(t)\equiv 1$. For example, if $n$ is even
and $p=2$ then $F_{n/2,2}(t)=f_{n/2,2}(t)\equiv 1$.

(2) For warped products ($h$ depends only on~$x$), $f_{k,p}(t)=F_{k,p}(t)=h^{\frac{n}{p}-k}(t)$.
\end{rem}
Let 
\[
\pi:I\times N\to N,\quad\pi(t,x)=x
\]
 be the natural projection. For $c\in[a,b)$, put $N_{c}:=\{c\}\times N$.
Let $i_{t}:N\to N_{t}\subset[a,b)\times N$ and $i_{c}:N\to N_{c}\subset[a,b)\times N$
be the natural immersions.

Every smooth $k$-form $\omega$ on $C_{a,b}^{h}N$ satisfies the
homotopy relations 
\begin{gather}
d_{N}\left(\int_{c}^{t}\omega\right)+\int_{c}^{t}d\omega=i{}_{t}^{*}\omega-i_{c}^{*}\omega,\label{eq:homN}\\
dS_{c}\omega+S_{c}d\omega=\omega-\pi^{*}i_{c}^{*}\omega.\label{eq:homM}
\end{gather}

Here $d_{N}$ stands for the exterior derivative on $N$ and $d$
designates the exterior derivative on $[a,b)\times N$.

The homotopy relations cannot be used automatically for $\omega\in\Omega_{q,p}^{k}(C_{a,b}^{h}N)$
because of the problem of the~existence of traces on submanifolds.
However, by Theorem~\ref{sm-cohom}, we can take only smooth
forms in all considerations concerning both reduced and nonreduced
$L_{q,p}$-cohomology.

For the reader's convenience, we repeat the classical proofs of~(\ref{eq:homN})
and~(\ref{eq:homM}).

Using the representation $\omega=\omega_{A}+dt\wedge\omega_{B}$,
we have 
\begin{gather*}
d\omega=d(\omega_{A}(t)+dt\wedge\omega_{B}(t))=dt\wedge\frac{\partial\omega_{A}}{\partial t}(t)+d_{N}\omega_{A}(t)-dt\wedge d_{N}\omega_{B}(t),\\
d_{N}\left(\int_{c}^{t}\omega\right)=d_{N}\left(\int_{c}^{t}\omega_{B}(\tau)d\tau\right)=\int_{c}^{t}d_{N}\omega_{B}(\tau)d\tau,\\
\int_{c}^{t}d\omega=\int_{c}^{t}\left(d\omega\right)_{B}(\tau)d\tau=\int_{c}^{t}\left(\frac{\partial\omega_{A}}{\partial t}(\tau)\right)d\tau-\int_{c}^{t}d_{N}\omega_{B}(\tau)d\tau.
\end{gather*}

Hence, 
\[
d_{N}\left(\int_{c}^{t}\omega\right)+\int_{c}^{t}d\omega=\int_{c}^{t}\left(\frac{\partial\omega_{A}}{\partial t}(\tau)\right)d\tau=\omega_{A}(t)-\omega_{A}(c)=i{}_{t}^{*}\omega-i_{c}^{*}\omega.
\]
 Similarly, 
\begin{gather*}
S_{c}\omega=\int_{c}^{t}\omega_{B}(\tau)d\tau;\quad dS_{c}\omega=d\left(\int_{c}^{t}\omega_{B}(\tau)d\tau\right)=dt\wedge\omega_{B}(t)+\int_{c}^{t}d_{N}\omega_{B}(\tau)d\tau,\\
S_{c}d\omega=\int_{c}^{t}d\omega=\int_{c}^{t}\left(\frac{\partial\omega_{A}}{\partial t}(\tau)\right)d\tau-\int_{c}^{t}d_{N}\omega_{B}(\tau)d\tau.
\end{gather*}

Therefore, 
\[
dS_{c}\omega+S_{c}d\omega=dt\wedge\omega_{B}(t)+\int_{c}^{t}\left(\frac{\partial\omega_{A}}{\partial t}(\tau)\right)d\tau=dt\wedge\omega_{B}(t)+\omega_{A}(t)-\omega_{c}(t)=\omega-i_{c}^{*}\omega.
\]

\begin{lem}
\label{lem:asymp} 
Suppose that $g$ is a function locally integrable
on $[a,b)$, \linebreak $\int_{a}^{b}|g(t)|dt=\infty$, and $\omega\in C^{\infty}L_{p}^{k}(C_{a,b}^{h}N)$.
Then every set of full measure on $[a,b)$ contains a sequence $\{t_{j}\}$
that converges to $b$ and is such that $g(t_{j})\ne0$ and $\|i_{t_{j}}^{*}\omega\|_{L^{p}(N,\Lambda^{k})}=o([f_{k,p}(t_{j})]^{-1}|g(t_{j})|^{1/p})$
as $j\to\infty$. \end{lem}
\begin{proof}
Equality (\ref{eq:shortnorm}) implies: 
\[
\int_{a}^{b}\|i_{t}^{*}\omega\|_{L^{p}(N,\Lambda^{k})}^{p}f_{k,p}^{p}(t)dt\leq\int_{a}^{b}\int_{N}|\omega_{A}(t)|^{p}\, h^{n-kp}(t,x)dxdt\le\|\omega\|_{L^{p}(C_{a,b}^{h}N,\Lambda^{k})}^{p}<\infty,
\]
 which yields the lemma. \end{proof}
\begin{lem}
\label{ineq} Suppose that $\omega\in L_{p}^{k}(C_{a,b}^{h}N)$, $c\in[a,b)$,
$p\ge q>1$. Then the form $\int_{c}^{t}\omega$ is defined for each
$t\in[a,b)$ and belongs to $L^{q}(N,\Lambda^{k-1})$. Moreover, 
\[
\left\Vert \int_{c}^{t}\right\Vert \le|N|^{\frac{1}{q}-\frac{1}{p}}\left|\int_{c}^{t}f_{k-1,p}^{-p'}(\tau)d\tau\right|^{1/p'}.
\]
 The same holds for $c=b$ if 
\[
\int_{a}^{b}f_{k-1,p}^{-p'}(\tau)d\tau<\infty.
\]
 \end{lem}

\begin{proof}
By H\"older's inequality, 
\begin{multline*}
\left\Vert \int_{c}^{t}\omega\right\Vert _{L^{q}(N,\Lambda^{k-1})}\le|N|^{\frac{1}{q}-\frac{1}{p}}\left\Vert \int_{c}^{t}\omega\right\Vert _{L^{p}(N,\Lambda^{k-1})}\\
=|N|^{\frac{1}{q}-\frac{1}{p}}\biggl(\int_{N}\biggl|\int_{c}^{t}\omega_{B}(\tau)d\tau\biggr|_{N}^{p}dx\biggr)^{1/p}\le|N|^{\frac{1}{q}-\frac{1}{p}}\biggl(\int_{N}\biggl|\int_{c}^{t}|\omega_{B}(\tau)|_{N}d\tau\biggr|^{p}dx\biggr)^{1/p}\\
\le|N|^{\frac{1}{q}-\frac{1}{p}}\Biggl[\int_{N}\biggl|\int_{c}^{t}h^{-(\frac{n}{p}-k+1)p'}(x,\tau)d\tau\int_{c}^{t}|\omega_B(\tau)|_N^{p}\, h^{n-kp+p}(\tau,x)d\tau\biggr|dx\Biggr]^{1/p}\\
\le|N|^{\frac{1}{q}-\frac{1}{p}}\Biggl[\int_{N}\biggl|\int_{c}^{t}f_{k-1,p}^{-p'}(\tau,x)d\tau\int_{c}^{t}|\omega_B(\tau)|_N^{p}\, h^{n-kp+p}(\tau,x)d\tau\biggr|dx\Biggr]^{1/p}\\
\le|N|{}^{\frac{1}{q}-\frac{1}{p}}\left|\int_{c}^{t}f_{k-1,p}^{-p'}(\tau)d\tau\right|\|\omega\|_{L^{p}(C_{a,b}^{h}N,\Lambda^{k})}.
\end{multline*}
 This proves the lemma. \end{proof}
\begin{rem}
The proof shows that Lemma~\ref{ineq} actually holds in the case where
$N$ is not necessarily compact but has finite volume. 
\end{rem}

\section{The Main Results}

\subsection{Absolute reduced $L_{q,p}$-cohomology}

Using the results of the previous section, we prove

\begin{thm}
\label{thm:abs} 
Let $N$ be a~closed smooth $n$-dimensional Riemannian manifold and let $p\ge q>1$. If 
$$
I_{a,b}:= \int_a^b F_{k,p}^p (t) dt = \infty;
\quad
J_{\delta_0, b}:=\int_{\delta_0}^b f_{k,p}^p(\tau)\left(\int_{a}^{\tau} F_{k,p}^p(t)dt \right)^{-1}d\tau=\infty
$$
for some $\delta_0\in [a,b)$ then $\overline{H}_{q,p}^{k}(C_{a,b}^h N)=0$. 
\end{thm}

\begin{rem}
The condition ``$J_{\delta_0, b}=\infty$ {\em for some} $\delta_0\in (a,b)$'' is in~fact equivalent to 
``$J_{\delta_0, b}=\infty$ {\em for every} $\delta_0 \in (a,b)$''. In this connection, below we sometimes
write $J_b$ instead of~$J_{\delta_0, b}$. 
\end{rem}

\begin{proof}
Suppose that $\omega\in L_{p}^{k}(C_{a,b}^{h}N)$ is weakly differentiable
and $d\omega=0$. Therefore, $\omega\in\Omega_{p,p}^{k}(C_{a,b}^{h}N)$.
By Theorem~\ref{sm-cohom}, we may assume that $\omega$ is a smooth
form. If $\int_{a}^{\delta_0} f_{k,p}^{p}(\tau)\left(\int_{a}^{\tau}F_{k,p}^{p}(t)dt\right)^{-1}d\tau=\infty$
for $\delta_0\in (a,b)$ then the function $g(\tau)=f_{k,p}^{p}(\tau)\left(\int_{a}^{\tau}F_{k,p}^{p}(t)dt\right)^{-1}$
is not integrable on intervals of the form $(c,b)$, $\delta_0\le c<b$.
By Lemma~\ref{lem:asymp}, there exists a~sequence $\{\tau_{j}\}\subset(\delta_0,b)\subset (a,b)$
such that $\tau_{j}\to b$ and 
\[
\|i_{\tau_{j}}^{*}\omega\|_{L^{p}(N,\Lambda^{k})}=o\left(\left[\int_{a}^{\tau_{j}}F_{k,p}^{p}(t)dt\right]^{-1/p}\right)\quad\mbox{as \ensuremath{j\to\infty}}.
\]
 Consider the form 
\[
\omega_{j}=\begin{cases}
\omega & \text{on \ensuremath{[\tau_{j},b)},}\\
\pi^{*}i_{\tau_{j}}^{*}\omega & \text{on \ensuremath{[a,\tau_{j})}.}
\end{cases}
\]
 It is easy to verify that $\omega_{j}$ belongs to $\Omega_{q,p}^{k}(C_{a,b}^{h}N)$.
Since $d\omega=0$ and $d\pi i_{\tau_{j}}^{*}\omega=0$, we have $d\omega_{j}=0$.

We infer 
\[
\|\omega_{j}\|_{L^{p}(C_{a,b}^{h}N,\Lambda^{k})}^{p}\le\int_{a}^{\tau_{j}}F_{k,p}^{p}(t)dt\,\,\|i_{\tau_{j}}^{*}\omega\|_{L^{p}(N,\Lambda^{k})}^{p}+\|\omega\|_{L^{p}(C_{\tau_{j},b}^{h}N,\Lambda^{k})}^{p}.
\]
 Here $C_{\tau_j,b}^{h}N$ is the product manifold $[\tau_j,b)\times N$
endowed with the metric induced from~$C_{a,b}^{h}N$.

Therefore, $\|\omega_{j}\|_{L^{p}(C_{a,b}^{h}N,\Lambda^{k})}\to0$
as $j\to\infty$; moreover, the form $\omega-\omega_{j}$ is equal
to $0$ on $[\tau_{j},b)\times N$. Hence, $S_{\tau_{j}}(\omega-\omega_{j})\in\Omega_{q,p}(C_{a,b}^{h}N)$
and $dS_{\tau_{j}}(\omega-\omega_{j})=\omega-\omega_{j}$. Thus, the
cocycle $\omega$ is zero in the reduced cohomology $\overline{H}_{q,p}^{k}(C_{a,b}^{h}N)$. 
\end{proof}

From Remark \ref{rem: absolute} and Theorem~\ref{thm:abs} we obtain

\begin{thm}\label{cor:special-abs}
Let $N$ be a~closed smooth $n$-dimensional Riemannian manifold.
If $p\ge q>1$, $b=\infty$ and $\frac{n}{p}$
is an integer then \hbox{$\overline{H}_{q,p}^{\frac{n}{p}}(C_{a,\infty}^{h}N)=0$}. 

In particular, if $1<q\le 2$ and $n$ is even then
\hbox{$\overline{H}_{q,2}^{\frac{n}{2}}(C_{a,\infty}^{h}N)=0$}.
\end{thm}

\begin{proof}
In this case, $I_{a,\infty}=J_{a,\infty}\!=\!\infty$, and by Theorem~\ref{thm:abs} 
$\overline{H}_{q,p}^{\frac{n}{p}}(C_{a,\infty}^{h}N)\!=\!0$.
\end{proof}

\begin{rem}
For a~warped cylinder $C_{a,b}^h N$, we have 
$f_{k,p}^{p}(\tau)=F_{k,p}^{p}(\tau)=h^{n-kp}(\tau)$.
Therefore, if 
$$
\int_a^b h^{n-kp} (t) dt = \infty
$$
then 
\begin{multline*}
J_{\delta_0, b}=\int_{\delta_0}^b h^{n-kp}(\tau) \left(\int_{a}^{\tau} h^{n-kp}(t)\, dt \right)^{-1} d\tau
=\int_{\delta_0}^b  \frac{d}{d\tau} \log \left( \int_a^\tau h^{n-kp}(t) dt \right) d\tau
\\
=\lim_{\tau\nearrow b} \left\{ \log \left( \int_a^\tau h^{n-kp}(t) dt \right)
- \log \left( \int_a^{\delta_0} h^{n-kp}(t) dt \right) \right\} = \infty,
\end{multline*}
and the result of Theorem \ref{thm:abs} coincides with the corresponding
result in~\cite{GKSh90_1}. 
\end{rem}

\subsection{Relative reduced $L_{q,p}$-cohomology}

Here we prove a sufficient vanishing condition for $\overline{H}_{q,p}^{k}(C_{a,b}^{h}N,N_{a})$,
where $N_{a}=\{a\}\times N$.

\begin{thm}
\label{thm:rel} 
Let $N$ be a~closed smooth $n$-dimensional Riemannian manifold.
Assume that $p\ge q>1$, 
$$
\tilde{I}_{a,b}:= \int_a^b f_{k-1,p}^{-p'} (t)dt = \infty,
$$ 
and the integral
$$
A_{\delta_0,b}:= 
 \int_{\delta_0}^b \frac{F_{k-1,p}^p(\tau)}{ f_{k-1,p}^{pp'}(\tau) }
\biggl( \int_a^\tau f_{k-1,p}^{-p'}(t) dt\biggr)^{-1} 
\biggl|\log\biggl(\int_a^\tau f_{k-1,p}^{-p'}(t) \, dt\biggr)\biggr|^{-p} \, d\tau
$$
is finite for some $\delta_0\in [a,b)$. Then $\overline{H}_{q,p}^k(C_{a,b}^h N, N_a)=0$.
\end{thm}

\begin{proof}
Suppose that $\omega\in C^{\infty}\Omega_{p,p}^{k}(C_{a,b}^{h}N,N_{a})$
and $d\omega=0$. The form $\omega$ is the limit in $\Omega_{p,p}^{k}(C_{a,b}^{h}N)$
of a sequence $\omega_{j}$ of smooth forms each of which is equal to zero
on some neighborhood of $N_{a}$.

As above, for $a\le c<d\le b$ denote by $C_{c,d}^{h}N$ the product
$[c,d)\times N$ with the metric induced from $C_{a,b}^{h}N$. For
every $e\in(a,b)$, Lemma \ref{ineq} implies that the operator 
$S_{a}:\Omega_{p,p}^{k}(C_{a,b}^{h}N)\to\Omega_{q,p}^{k-1}(C_{a,e}^{h}N)$
is bounded. Hence, $S_{a}\omega_{j}\to S_{a}\omega$ in $\Omega_{q,p}^{k-1}(C_{a,e}^{h}N)$.
Each of the forms $S_{a}\omega_{j}$ vanishes on some neighborhood of~$N_{a}$. Therefore, 
$S_{a}\omega\in\Omega_{q,p}^{k-1}(C_{a,e}^{h}N,N_{a})$ for all $e\in(a,b)$. Then the fact 
that $i_{a}^{*}=0$ and relation~(\ref{eq:homM}) give the equality $dS_{a}\omega=\omega$.

Consider the functions 
\begin{gather*}
I(\tau)=\int_{a}^{\tau}f_{k-1,p}^{-p'}(t)dt,\quad\varphi(\tau)=\log|\log I(\tau)|;\\
\varphi_{\delta}(\tau)=\begin{cases}
1 & \text{if \ensuremath{\tau\le\delta},}\\
\min(1,\max(1+\varphi(\delta)-\varphi(\tau),0)) & \text{if \ensuremath{\tau\ge\delta}.}
\end{cases}
\end{gather*}

The function $\varphi(t)$ is defined for $\tau$ sufficiently close
to $b$ and $\varphi_{\delta}(t)$ also exists only for $\delta$
that are sufficiently close to b.

Since $d(\varphi_{\delta}S_{a}\omega)=d\varphi_{\delta}\wedge S_{a}\omega+\varphi_{\delta}\omega$,
it follows that 
\begin{multline*}
\|d(\varphi_{\delta}S_{a}\omega)-\omega\|_{L^{p}(C_{a,b}^{h}N,\Lambda^{k})}=\|d(\varphi_{\delta}S_{a}\omega)-\omega\|_{L^{p}(C_{\delta,b}^{h}N,\Lambda^{k})}\\
\le\|d\varphi_{\delta}\wedge S_{a}\omega\|_{L^{p}(C_{\delta,b}^{h}N,\Lambda^{k})}+\|(\varphi_{\delta}-1)\omega\|_{L^{p}(C_{\delta,b}^{h}N,\Lambda^{k})}.
\end{multline*}
 Since $|\varphi_{\delta}-1|\le1$, we have 
\[
\|(\varphi_{\delta}-1)\omega\|_{L^{p}(C_{\delta,b}^{h}N,\Lambda^{k})}\le\|\omega\|_{L^{p}(C_{\delta,b}^{h}N,\Lambda^{k})}.
\]

By~(\ref{eq:norm}) and Lemma~\ref{ineq} for $p=q$, for $\delta$
sufficiently close to $b$ we infer 
\begin{multline*}
\|d\varphi_{\delta}\wedge S_{a}\omega\|_{L^{p}(C_{\delta,b}^{h}N,\Lambda^{k})}^{p}=\int_{\delta}^{b}h^{n-kp+p}(\tau,x)\biggl|\frac{d\varphi_{\delta}}{d\tau}\biggr|^{p}\biggl|\int_{a}^{\tau}\omega\biggr|_{N}^{p}dx\, d\tau\\
\le\int_{\delta}^{b}h^{n-kp+p}(\tau,x)\biggl|\frac{d\varphi_{\delta}}{d\tau}\biggr|^{p}\biggl(\int_{a}^{\tau}f_{k-1,p}^{-p'}(t)dt\biggr)^{\frac{p}{p'}}\, d\tau\,\|\omega\|_{L^{p}(N,\Lambda^{k})}^{p}\\
\le\int_{\delta}^{b}\!\! F_{k-1,p}^{p}(\tau)f_{k-1,p}^{-pp'}(\tau)\biggl(\int_{a}^{\tau}\!\! f_{k-1,p}^{-p'}(t)dt\! \biggr)^{-1}
\biggl|\log\biggl(\int_{a}^{\tau}\!\! f_{k-1,p}^{-p'}(t)\, dt\biggr)\biggr|^{-p}\!\! d\tau\,\|\omega\|_{L^{p}(N,\Lambda^{k})}^{p}.
\end{multline*}

By hypothesis, the last quantity vanishes as $\delta\to b$. Thus,
$d(\varphi_{\delta}S_{a}\omega)\to\omega$ as $\delta\to b$. The
function $\varphi_{\delta}$ is equal to zero in some neighborhood
of $b$. Therefore, $\varphi_{\delta}S_{a}\omega\in\Omega_{q,p}^{k}(C_{a,b}^{h}N,N_{a})$.
This shows that ${\overline{H}}_{q,p}^{k}(C_{a,b}^{h}N,N_{a})=0$. 
\end{proof}

\begin{rem}
The paper \cite{GKSh90_1} contains the following assertion for a~warped
cylinder $C_{a,b}^{h}N$ \cite[Theorem~2]{GKSh90_1}:
\smallskip

\textit{If $\int\limits _{a}^{b}h^{-(\frac{n}{p}-k+1)p'}(t)dt<\infty$
then $\overline{H}_{p,p}^{k}(C_{a,b}^{h}N,N_{a})=0$. } 
\smallskip

Suppose that $\int\limits_{a}^{b}h^{-(\frac{n}{p}-k+1)p'}(t)dt<\infty$. We infer
\begin{multline*}
\int_{\delta}^{b} h^{-(\frac{n}{p}-k+1)p'}(\tau)
\biggl(\int_{a}^{\tau} h^{-(\frac{n}{p}-k+1)p'}(t) dt\biggr)^{-1}\biggl|\log\biggl(\int_{a}^{\tau}
h^{-(\frac{n}{p}-k+1)p'}(t)\, dt\biggr)\biggr|^{-p}\, d\tau
\\
= \frac{1}{p-1} \int_\delta^b \frac{d}{d\tau} 
\left[ \log \left( \int_a^\tau h^{-(\frac{n}{p}-k+1)p'}(t)\, dt \right)  \right]^{1-p} d\tau
\\
= \frac{1}{p-1} \left[ \log \left( \int_a^\delta h^{-(\frac{n}{p}-k+1)p'}(t) dt \right) \right]^{1-p}
\to 0 \ \text{as} \ \delta\to b.
\end{multline*}

Thus, Theorem~\ref{thm:rel} generalizes this assertion to twisted cylinders~$C_{a,b}^{f}N$
with $N$~closed. 
\end{rem}
\vspace{2mm}

Remark \ref{rem: absolute} and Theorem~\ref{thm:rel} imply

\begin{thm}\label{cor:special-rel}  
Let $N$ be a~closed smooth $n$-dimensional Riemannian manifold.                                               
If $p\ge q>1$, $b=\infty$, and $\frac{n}{p}$ is an integer 
then \hbox{$\overline{H}_{q,p}^{\frac{n}{p}+1}(C_{a,b}^{h}N,N_{a})=0$}. 

In particular, if $1<q\le 2$ and $n$ is even then
\textup{$\overline{H}_{q,2}^{\frac{n}{2}+1}(C_{a,b}N,N_{a})=0$.}
\end{thm}

\begin{proof}
In this case, $A_{\delta_0,b}=\int_{\delta_0}^b(\tau-a)^{-1}\left(\log(\tau-a)\right)^{-p}d\tau<\infty$
for any $a,b$ and, by Theorem \ref{thm:rel}, $\overline{H}_{q,p}^{k}(C_{a,b}^{h}N,N_{a})=0$. 
\end{proof}

\section{Asymptotic Twisted Cylinders}

\begin{defn}
We refer to a pair $(M,X)$ consisting of an $m$-dimensional manifold $M$ and
an $n$-dimensional compact submanifold~$X$ with boundary as an {\it asymptotic twisted cylinder}
$AC_{a,b}^{h}\partial X$ if $M\setminus X$ is bi-Lipschitz diffeomorphic equivalent 
to the~twisted cylinder $C_{a,b}^{h}\partial X$. 
\end{defn}

Theorem \ref{thm:rel} readily implies

\begin{thm}\label{thm:app-rel}
Let $(M,X)=AC_{a,b}^{h}\partial X$ be an~asymptotic twisted cylinder. Assume that 
$p\!\ge\! q\!>\!1$, 
$$
\int_a^b f_{k-1,p}^{-p'} (t)dt = \infty,
$$ 
and the integral
$$
A_{\delta_0,b}:= 
 \int_{\delta_0}^b \frac{F_{k-1,p}^p(\tau)}{ f_{k-1,p}^{pp'}(\tau)} 
\biggl( \int_a^\tau f_{k-1,p}^{-p'}(t) dt\biggr)^{-1} 
\biggl|\log\biggl(\int_a^\tau f_{k-1,p}^{-p'}(t) \, dt\biggr)\biggr|^{-p} \, d\tau
$$
is finite for some $\delta_0\in (a,b)$. 

Then $\overline{H}_{q,p}^{k}(M,X)=0$.
\end{thm}

\begin{proof}
Note that bi-Lipschitz diffeomorphisms preserve $L_p$ and $L_q$. Moreover, extension
by zero gives a~topological isomorphism between the relative spaces $W_{r,s}(C_{a,b}^h \partial X, (\partial X)_a)$
and $W_{r,s}(M,X)$ for all~$r,s$. This gives topological isomorphisms 
$$
H^*_{r,s}(M,X)\cong H^*_{r,s}(C_{a,b}^h \partial X, (\partial X)_a); \quad
\overline{H}^*_{r,s}(M,X)\cong \overline{H}^*_{r,s}(C_{a,b}^h \partial X, (\partial X)_a)
$$
for all~$r$, $s$. The theorem now follows from Theorem~\ref{thm:rel}.
\end{proof}

\begin{rem}
In the definitions of the functions $F_{k-1,p}$ and $f_{k-1,p}$, \,$n=m-1$.

To obtain a~version of this theorem for $\overline{H}_{q,p}^{k}(M)$, we will need the~exact
sequences of a~pair for $L_{q,p}$-cohomology. 
\end{rem}

Recall that an exact sequence 
\begin{equation}\label{seqcom-1}
0\to A\overset{\varphi}{\to}B\overset{\psi}{\to}C\to 0 
\end{equation}
 of cochain complexes of vector spaces is called an {\em exact sequence
of Banach complexes} if $A$, $B$, $C$ are Banach complexes and
all mappings $\varphi^{k}$, $\psi^{k}$ are bounded linear operators.

A short exact sequence~(\ref{seqcom-1}) of Banach complexes induces
an exact sequence in cohomology 
\[
\dots\longrightarrow H^{k-1}(C)\overset{\partial^{k-1}}{\longrightarrow}H^{k}(A)\overset{\varphi^{*k}}{\longrightarrow}
H^{k}(B)\overset{\psi^{*k}}{\longrightarrow}H^{k}(C)\longrightarrow\dots
\]
 with all operators $\partial$, $\varphi^{*}$, $\psi^{*}$ bounded
and a~sequence in reduced cohomology 
\begin{equation}
\dots\longrightarrow\overline{H}^{k-1}(C)\overset{\overline{\partial}^{k-1}}{\longrightarrow}
\overline{H}^{k}(A)\overset{\overline{\varphi}^{*k}}{\longrightarrow}\overline{H}^{k}(B)
\overset{\overline{\psi}^{*k}}{\longrightarrow}\overline{H}^{k}(C)\longrightarrow\dots,\label{seqredcom-1}
\end{equation}
 where all operators $\overline{\partial}$, $\overline{\varphi}^{*}$,
$\overline{\psi}^{*}$ are bounded and the composition of any two
consecutive morphisms is equal to zero. Sequence~(\ref{seqredcom-1})
is in general not exact but exactness at its particular terms
can be guaranteed by some additional assumptions. For example, we
have (see \cite[Theorem~1\,(1)]{GKSh90_1}):

\begin{prop}
\label{prop:exact-piece} 
Given an~exact sequence~{\rm(\ref{seqcom-1})}
of Banach complexes, if \linebreak $H^{k}(C)_{0}=0$ and $\dim\partial^{k-1}(H^{k-1}(C))<\infty$
then the sequence 
$$
\overline{H}^{k-1}(C)\overset{\overline{\partial}^{k-1}}{\longrightarrow}
\overline{H}^{k}(A)\overset{\overline{\varphi}^{*k}}{\longrightarrow}\overline{H}^{k}(B)
\overset{\overline{\psi}^{*k}}{\longrightarrow}\overline{H}^{k}(C)
$$
is exact. 
\end{prop}

Let $(M,X)=AC_{a,b}^{h}\partial X$ be an~asymptotic twisted cylinder.

Consider the short exact sequence of Banach complexes 
\begin{equation}
0\to C^{\infty}\Omega_{\mathbb{P}}(M,X)\overset{j}{\to}C^{\infty}\Omega_{\mathbb{P}}(M)
\overset{i}{\to}C^{\infty}\Omega_{\mathbb{P}}(X)\to 0.\label{cinfw-1}
\end{equation}
 In all the three complexes of~(\ref{cinfw-1}), 
\[
C^{\infty}\Omega_{\mathbb{P}}=\begin{cases}
C^{\infty}\Omega_{q,q} & \text{ if }q\le k-1;\\
C^{\infty}\Omega_{q,p} & \text{ if }q=k;\\
C^{\infty}\Omega_{p,p} & \text{ if }q\ge k+1;
\end{cases}
\]
$i$ is the natural embedding, and $j$ is the restriction of forms.
The cohomology of $C^{\infty}\Omega_{\mathbb{P}}(X)$ is simply the
de Rham cohomology of $X$. 

\begin{thm}\label{thm:app-abs} 
Let $(M,X)=AC_{a,b}^{h}\partial X$ be an~asymptotic twisted cylinder. Assume that \textup{$p\ge q>1$}, $H^{k}(X)=0$,
$$
\int_a^b f_{k-1,p}^{-p'} (t)dt = \infty
$$ 
and the integral
$$
A_{\delta_0,b}:= 
 \int_{\delta_0}^b \frac{F_{k-1,p}^p(\tau)}{f_{k-1,p}^{pp'}(\tau)} 
\biggl( \int_a^\tau f_{k-1,p}^{-p'}(t) dt\biggr)^{-1} 
\biggl|\log\biggl(\int_a^\tau f_{k-1,p}^{-p'}(t) \, dt\biggr)\biggr|^{-p} \, d\tau
$$
is finite for some $\delta_0\in (a,b)$. 

Then $\overline{H}_{q,p}^{k}(M)=0$.
\end{thm}

\begin{proof}
Using Theorem~\ref{sm-cohom} and Proposition~\ref{prop:exact-piece},
we obtain the exact sequence of reduced cohomology spaces 
\begin{equation}\label{isom-1}
0\longrightarrow\overline{H}_{q,p}^{k}(M,X)\overset{\cong}{\longrightarrow}
\overline{H}_{q,p}^{k}(M)\longrightarrow0.
\end{equation}
 
By Theorem~\ref{thm:app-rel}, $\overline{H}_{q,p}^{k}(M,X)=0$ if
$p\ge q>1$, $\int_a^b f_{k-1,p}^{-p'} (t)dt = \infty$, and \mbox{$A_{\delta_0,b}<\infty$} 
for some $\delta_0\in (a,b)$. 
Hence, $\overline{H}_{q,p}^{k}(M)=0$.
\end{proof}

Theorem~\ref{cor:special-rel} and the exact sequence~(\ref{isom-1}) readily imply

\begin{thm}\label{cor:app-abs-special}
Let $(M,X)=AC_{a,b}^{h}\partial X$ be an~asymptotic twisted cylinder with $\dim X=n$. 
If $p\ge q>1$, $\frac{n}{p}$ is an~integer, and 
$H^{\frac{n}{p}}(X)=H^{\frac{n}{p}+1}(X)=0$ then $\overline{H}_{q,p}^{\frac{n}{p}+1}(M)=0$. 

In particular, if $1<q\le 2$, $n$ is even, and $H^{\frac{n}{2}}(X)=H^{\frac{n}{2}+1}(X)=0$
then \textup{$\overline{H}_{q,2}^{\frac{n}{2}+1}(M)=0$.}
\end{thm}

Using the duality theorem (Theorem~\ref{thm:duality}), we now reformulate this result 
for cohomology with compact support (interior cohomology). By the classical Poincar\'e 
duality, $H^{\frac{n}{p}}(X)=H^{n-\frac{n}{p}}(X)=H^{\frac{n}{p'}}(X)$ and 
$H^{\frac{n}{p}+1}(X)=H^{n-\frac{n}{p}-1}(X)=H^{\frac{n}{p'}-1}(X)$ (where 
$\frac{1}{p}+\frac{1}{p'}=1$). If  $p\ge q>1$ then $q' \ge p'> 1$. 

Theorem~\ref{cor:app-abs-special} and Theorem ~\ref{thm:duality} immediately imply

\begin{thm}\label{cor:app-int-special}
Let $(M,X)=AC_{a,b}^{h}\partial X$ be an asymptotic twisted cylinder with $\dim X=n$. 
If $q\ge p>1$, $\frac{n}{p}$ is an~integer, and 
$H^{\frac{n}{p}-1}(X)=H^{\frac{n}{p}}(X)=0$ then $\overline{H}_{p,q;0}^{\frac{n}{p}}(M)=0$. 

In particular, if $q\!\ge\! 2$, $n$ is even, and $H^{\frac{n}{2}-1}(X)\!=\!H^{\frac{n}{2}}(X)\!=0$
then $\overline{H}_{2,q;0}^{\frac{n}{2}}(M)\!=\!0$.
\end{thm}

\smallskip

Recall that if a~manifold $Y$ is complete then $\overline{H}_p^k(Y)=\overline{H}_{p;0}^k(Y)$. We have

\begin{cor} \label{cor:p=q}
Let $(M,X)=AC_{a,b}^{h}\partial X$ be an asymptotic twisted cylinder with $\dim X=n$ 
and let $M$ be a~complete manifold.

If $p >1$, $\frac{n}{p}$ is an~integer, and $H^{\frac{n}{p}}(X)=H^{\frac{n}{p}+1}(X)=0$ then 
$\overline{H}_{p}^{\frac{n}{p}+1}(M)=\overline{H}_{p'}^{\frac{n}{p'}}(M)=0$. 

In particular, if $n$ is even and $H^{\frac{n}{2}}(X)=H^{\frac{n}{2}+1}(X)=0$
then \textup{$\overline{H}_{2}^{\frac{n}{2}}(M)=\overline{H}_{2}^{\frac{n}{2}+1}(M)=0$.} 
\end{cor}

\smallskip
Suppose that $(M,g)$ is an $m$-dimensional Cartan--Hadamard
manifold with sectional curvature $K\leq-1$ and Ricci curvature $Ric\ge-(1+\epsilon)^{2}(m-1)$
and $B$ is a~unit closed geodesic ball in $(M,g)$. Then, as was observed in the proof
of~\cite[Lemma~3.1]{GT2009}, the pair $(M,B)$ is an~asymptotic twisted cylinder endowed 
with the Riemannian metric 
\[
g=dr^{2}+g_{r}.
\]
Here $g_{r}$ is a Riemannian metric on the sphere $\{r\}\times\mathbb{S}^{m-1}$.
Since $B$ is a~topological ball, $H^{k}(B)=0$ for any $k=1,2,\dots,m$.
Theorem~\ref{cor:app-abs-special} gives: 

\begin{prop}\label{prop:Cartan_Hadamard-abs-special} 
Let $(M,g)$ be an $m$-dimensional Cartan--Hadamard manifold {\rm(}that is, 
a complete simply-connected Riemannian manifold of nonpositive sectional 
curvature{\rm)} with sectional curvature $K\!\!\leq\!-1$ and Ricci curvature 
\mbox{$Ric\!\ge\!\!-(1\!+\!\epsilon)^{2}(m\!-\!1\!)$}. If $p\ge q>1$ and $\frac{m-1}{p}$ 
is an integer then $\overline{H}_{q,p}^{\frac{m-1}{p}+1}(M,g)=0$.

In particular,
\begin{itemize} 
\item if $1<q\le 2$ and $m$ is odd then \textup{$\overline{H}_{q,2}^{\frac{m-1}{2}+1}(M,g)=0$;}
\item if $p >1$ and $\frac{m-1}{p}$ is an~integer then $\overline{H}_{p}^{\frac{m-1}{p}+1}(M,g)=\overline{H}_{p'}^{\frac{m-1}{p'}}(M,g)=0;$ 
\item if  $m$ is odd then \textup{$\overline{H}_{2}^{\frac{m-1}{2}}(M,g)=\overline{H}_{2}^{\frac{m+1}{2}}(M,g)=0$.}
\end{itemize}

\end{prop}

\section{Examples}
Let us consider some examples.  Unfortunately, our methods only 
make it possible to make conclusions mainly about warped cylinders and not 
about twisted cylinders. 

Below the symbols $C$, $C_{1}$, $C_{2}$ stand for positive constants.

\begin{prop}\label{exp+}
Suppose that $a\ge0$, $b=\infty$, and $C_{1}e^{s_{1}t}\le h(t,x)\le C_{2}e^{s_{2}t}$
($s_{2}\ge s_{1}\ge0$). Then 

(1) $\overline{H}_{q,p}^{k}(C_{a,\infty}^{h}N)=0$ in each of the following cases:

(1a) $p\ge q>1$, $k=\frac{n}{p}$;

(1b) $p\ge q>1$, $s_1=s_2$, $k<\frac{n}{p}$;

(2) $\overline{H}_{q,p}^{k}(C_{a,\infty}^{h}N,N_{a})=0$ in each of the following cases:

(2a) $p\ge q>1$, $k=\frac{n}{p}+1$;

(2b) $p\ge q>1$, $k>\frac{n}{p}+1+\frac{1}{pp' s}$.

\end{prop}
      
\begin{proof}
(1) Assume first that $n-kp>0$.

We have 
$$
I= \int_0^\infty F_{k,p}^p(t) dt =\infty;
$$
the integral $J_{0,\infty}$ is infinite if so is the integral
$$
\int_0^\infty e^{-(s_2-s_1)(n-kp)t} dt
$$
We infer 
$$
\int_0^{\tau}e^{s_{2}(n-kp)}dt=C_{1}e^{s_{2}(n-kp)\tau}-C
$$
Further, 
\[
\int_{0}^{\infty}e^{-(s_{2}-s_{1})(n-kp)\tau}d\tau
\]
is infinite if $(s_1-s_2)(n-kp)\ge 0$. This leads to $s_1=s_2$.

The case of $p\ge q>1$, $k=\frac{n}{p}$ stems from Theorem~\ref{cor:special-abs}.

(2) Item~(2a) immediately follows from Theorem~\ref{cor:special-rel}. 

Consider case~(2b). Since $k>\frac{n}{p}+1$, we have $(k-1)p-n>0$.  Then
$$
{\tilde I}_{0,\infty}=\int_0^\infty f_{k-1,p}^{-p'}(t) dt =\infty.
$$
Fix $\delta_0>0$ so that $\int_0^\infty f_{k-1,p}^{-p'}(t) dt > 1$. The integrand
in~$A_{\delta_0,\infty}$ is at most some function~$\gamma(\tau)$ which, as $\tau\to\infty$,
is equivalent to
$$ 
C e^{(s_2-s_1)((k-1)p-n)\tau} \tau^{-s_1 ((k-1)p-n)p'}, 
$$
and the integral from~$\delta_0$ to~$\infty$ of this expression is finite if and only if 
$s_1=s_2$ and $-s_1 ((k-1)p-n)p'<-1$, which is sufficient for the finiteness of~$A_{\delta_0,\infty}$
and is just the hypothesis of~(2b).
\end{proof}

The following situation can also be considered in a~straightforward manner
with the use of Theorems~\ref{cor:special-abs} and~\ref{cor:special-rel}:

\begin{prop}\label{t-inf}
Suppose that $a\ge 1$, $b=\infty$, and $C_{1}t^{s_{1}}\le h(t,x)\le C_{2}t^{s_{2}}$
($s_{2}\ge s_{1}\ge0$). Then 

(1) $\overline{H}_{q,p}^{k}(C_{a,\infty}^{h}N)$ is zero in each of the following cases:

(1a) $p\ge q>1$, $k=\frac{n}{p}$;

(1b) $p\ge q>1$,  $s_1=s_2$, $k<\frac{n}{p}$;

(1c) $p\ge q>1$, $s_1=s_2$, $\frac{n}{p}<k\le \frac{n}{p}+\frac{1}{p s_1}$.

(2) $\overline{H}_{q,p}^{k}(C_{a,\infty}^{h}N,N_{a})$
is zero in each of the following cases:

(2a) $p\ge q>1$, $s_{1}=s_{2}$, $\frac{n}{p}+1-\frac{1}{p' s_1}\le k \le \frac{n}{p}+1$;

(2b) $p\ge q>1$, $k=\frac{n}{p}+1$;

\begin{rem}
(1) Propositions~\ref{exp+} and~\ref{t-inf} also hold for asymptotic twisted cylinders.

(2) The results are new even for $L^p$-cohomology $(p=q)$.
\end{rem}

\end{prop}

\end{document}